\documentclass{amsart}
\usepackage{graphicx,xcolor,booktabs,hyperref,tikz,multicol,mathtools,placeins,epstopdf,xspace}
\usepackage{siunitx}
\usepackage{subcaption} 
\DeclareGraphicsExtensions{.png,.pdf,.jpg,.eps}
\usepackage{xcolor}
\usepackage{amsmath}
\usepackage{booktabs}    
\usepackage{subcaption}  

\newcommand{\G}{\mathcal{G}}

\newtheorem{theorem}{Theorem}[section]
\newtheorem{lem}[theorem]{Lemma}

\theoremstyle{definition}

\title{Balanced spanning trees of the $2$-by-$N$ grid}

\author{Makenzie Gallagher}
\address{Computer Science Major, Saint Joseph's University, Philadelphia, PA,19026}
\email{cg10794149@sju.edu} 

\author{Kristopher Tapp}
\address{Department of Mathematics, Saint Joseph's University, Philadelphia, PA,19026}
\email{ktapp@sju.edu}


\begin{document}

\begin{abstract}
We obtain an exact formula for the probability that a uniformly random spanning tree of the $2$-by-$n$ square grid is ``balanced'' in the sense that it has an edge whose removal partitions its vertices into two sets of equal size.  We compute the exact limit of this probability as $n\rightarrow\infty$.
\end{abstract}
\maketitle

\section{Introduction}
Recent advances in computational redistricting have sparked interest in \emph{balanced} spanning trees.  Several algorithms have been used in court cases this decade to argue that voting maps are biased, including ReCom (\cite{DeFord_Duchin_Solomon_2020, Cannon_Duchin_Randall_Rule_2022}),  Forest ReCom (\cite{Autry_Carter_Herschlag_Hunter_Mattingly_2023}) and SMC (\cite{McCartan_Imai_2023}).  Each of these algorithms is capable of generating an ensemble of thousands or millions of random nonpartisan maps by repeatedly performing a step that involves drawing a random spanning tree of a graph and identifying a \emph{balanced cut edge} (an edge whose removal partitions the vertices into two sets of approximately equal weight), if one exists.  A better theoretical understanding of this likelihood could help to understand the performance of these algorithms and improve upon them.

A significant step was recently taken in this direction.  The authors of~\cite{Charikar_Liu_Liu_Vuong_2023} proposed a provably fast-mixing algorithm for sampling not-necessarily-balanced partitions of the vertices of a graph into connected pieces.  In the case of grids, they conjectured that a sufficiently large portion of such a sample would be balanced.  In the redistricting context, this amounts to building an ensemble of voting maps whose districts are not even approximately equipopulous, and hoping that enough are balanced to yield a sufficiently large subsample of legal maps. Their conjecture was proven by the authors of~\cite{Cannon_Pegden_Tucker-Foltz_2024}, who established that for fixed $k$, the portion of spanning trees of the $n$-by-$m$ square grid that are $k$-\emph{balanced} is at least $1/\text{poly}(m,n)$, where ``$k$-balanced'' means there exist $k-1$ edges whose removal partitions the vertices into $k$ sets of equal sizes.  Very recent work has expanded these results and proof techniques~\cite{Cannon_Pegden_Tucker-Foltz_2025},~\cite{Gillman_Platnick_Randall_2025},~\cite{Chen_Munagala_Sankar_2025}.

The goal of this paper is to obtain exact results in the simplest nontrivial case: the $2$-by-$n$ grid, which we denote as $\G_n$.  Let $T_n$ denote the number of spanning trees of $\G_n$.  Let $S_n$ denote the number of \emph{balanced} spanning trees; i.e., trees that contain an edge (called a \emph{balanced cut edge}) whose removal partitions the vertices into two sets of equal size.  Our main result is an exact formula for the ratio $\frac{S_n}{T_n}$, from which the following limits follow.
\begin{theorem}\label{T:main}
\hspace{1in}
\begin{enumerate}
\item \textbf{(Odd case)} 
$\lim_{m\rightarrow\infty} \frac{S_{2m+1}}{T_{2m+1}}= \frac{3+\sqrt{3}}{9}\approx 0.52578$
\item \textbf{(Even case)}
$\lim_{m\rightarrow\infty} \frac{S_{2m}}{T_{2m}}= \frac{1+4\sqrt{3}}{6\sqrt{3}}\approx 0.76289$.
\end{enumerate}
\end{theorem}

In Section~\ref{S:conclusion}, we additionally consider the probability that a spanning tree of $\G_n$ is balanced if, instead of choosing $T$ uniformly (that is, choosing it from the UST ``uniform spanning tree'' distribution), it is constructed as a minimal spanning tree with respect to random edge weights (that is, it is drawn from the MST ``minimal spanning tree'' distribution).  Recent theoretical results comparing the UST and MST distributions are found in~\cite{Babson_Duchin_Iseli_Poggi-Corradini_Thurston_Tucker-Foltz_2024} and~\cite{Tapp_2024}.  Surprisingly, for the even choices of $n$ that we computed, the MST probabilities are greater than the UST probabilities.
\section*{Acknowledgments}
This project was supported by the Summer Scholars Program for undergraduate research at Saint Joseph's University.    
\section{Related work}
The literature cited in the previous section motivates and/or addresses general questions about the probability that a random spanning tree is balanced.  In this section, we mention some results in the literature about the structure of $\G_n$ (the $2$-by-$n$ grid).  

The author of~\cite{Raff_2008} derived the following recursive formula for the number of possible spanning trees over a $2$-by-$n$ grid:
\begin{equation} \label{E:Raff}
T_{n+2} = 4\cdot T_{n+1} - T_{n},
\end{equation}
and showed that the generating function of the sequence $T_n$ is
$\frac{x}{1-4x+x^2}$.  His recursive formula is the starting point of our work.

More recently, the authors of~\cite{Blanco_Zeilberger_2025} proved that the limit as $n\rightarrow\infty$ of the portion of the vertices of a uniformly random spanning tree of $\G_n$ which are leaves (degree $1$) equals
$\frac{14\sqrt{5}}{5}-6\approx 0.26$.
\section{Set up}
Let $\G^*_n$ denote the dual graph of $\G_n$.  Let $v_\infty$ denote the vertex of $\G^*_n$ that corresponds to the unbounded face of $\G_n$.  For example, Figure~\ref{F:dual} shows the vertices of $\G_{11}$ in red and the vertices of $\G^*_{11}$ in yellow ($v_\infty$ is not explicitly shown in the figure).

\begin{figure}[bht!]\centering
\includegraphics[width=3in]{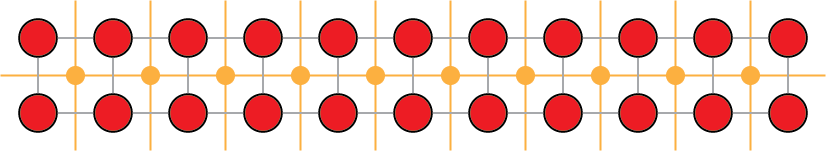}
\caption{$\G_{11}$ (in red) and $\G^*_{11}$ (in yellow).}\label{F:dual}
\end{figure}

Let $\mathcal{T}_n$ denote the set of spanning trees of $\G_n$, so $T_n=|\mathcal{T}_n|$.  By convention, we set $T_0=1$.  It is well known that the spanning trees of $\G_n$ correspond one-to-one with the spanning trees of $\G_n^*$.  For any $T\in\mathcal{T}_n$, let $T^*$ denote the corresponding spanning tree of $\G_n^*$.  With respect to the natural one-to-one correspondence between the edges of $\G_n$ and the edges of $\G^*_n$, the edges of $T^*$ are exactly the ones corresponding to edges of $\G_n$ that do not belong to $T$, as illustrated in Figure~\ref{F:tree_and_dual}.

Removing an edge of $T$ corresponds to adding an edge to $T^*$, which creates a unique cycle $\gamma$ in $T^*$.  Note that all cycles of $\G_n^*$ contain $v_\infty$, so we can regard $\gamma$ as a path from $v_\infty$ to $v_\infty$.  Figure~\ref{F:dual_tree} shows the cycle $\gamma$ induced by removing the highlighted edge of $T$.

\begin{figure}[bht!]\centering
\includegraphics[width=3in]{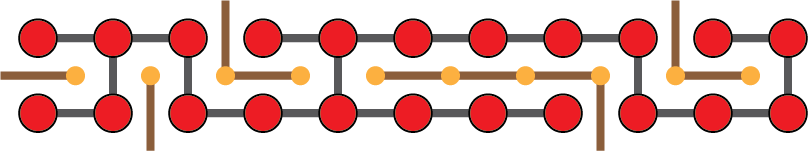}
\caption{A spanning tree $T$ of $\G_{11}$ (grey) and the corresponding spanning tree $T^*$ of $\G^*_{11}$ (brown).}\label{F:tree_and_dual}
\end{figure}

\begin{figure}[bht!]\centering
\includegraphics[width=3in]{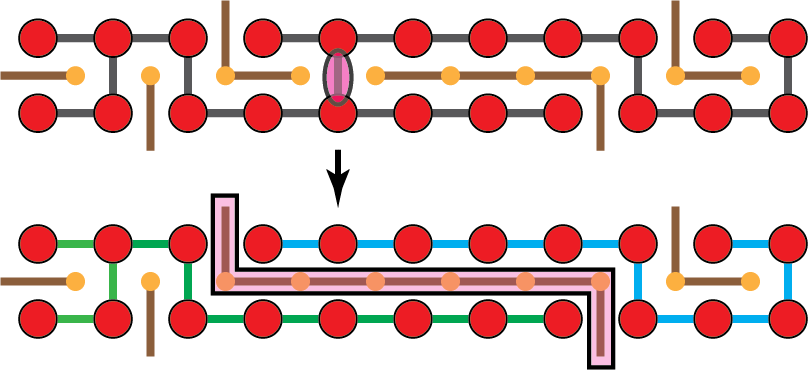}
\caption{Removing the highlighted edge of $T$ adds an edge to $T^*$, which create a loop at $v_\infty$. }\label{F:dual_tree}
\end{figure}

To foreshadow a proof in the next section, we note that the tree $T$ in Figures~\ref{F:tree_and_dual} and~\ref{F:dual_tree} is balanced, and that there is a family of other balanced trees that yield the same cycle $\gamma$.  More precisely, Figure~\ref{F:split_var} indicates that the portion of $T$ in each end block (highlighted yellow) can be replaced by any other spanning tree of $\G_3$, while the cut edge can be placed at any star -- these stars are the edges of $\G_{11}$ that correspond via duality to the edges of $\gamma$.
\begin{figure}[bht!]\centering
\includegraphics[width=3in]{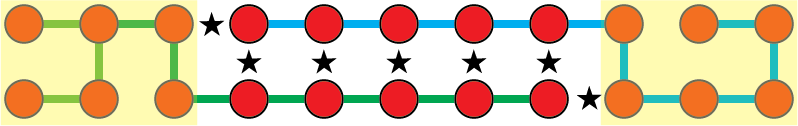}
\caption{Starting from the spanning tree $T$ of Figure~\ref{F:tree_and_dual}, an alternative balanced tree with the same $\gamma$ can be obtained by modifying $T$ in the two yellow end blocks and selecting the balanced cut edge among the locations marked by stars.}\label{F:split_var}
\end{figure}
\section{Counting balanced trees}
Let $\mathcal{S}_n\subset\mathcal{T}_n$ denote the set of balanced spanning trees of $\G_n$, so $S_n=|\mathcal{S}_n|$.  The goal of this section is to derive the following explicit formula for $S_n$.
\begin{lem}\label{L:even_odd} Let $m\geq 1$ be an integer.
\begin{enumerate}
\item \textbf{(Odd case):} If $n=2m+1$, then
$$S_{n} = n + \sum_{i=0}^{m-1} (6+4i)\cdot T^2_{m-i}.$$
\item \textbf{(Even case):} If $n=2m$, then
$$S_{n} = n + 2\cdot  T^2_m + \sum_{i=1}^{m-1} (4+4i)\cdot T^2_{m-i}.$$
\end{enumerate}
\end{lem}

\begin{proof} As illustrated in Figure~\ref{F:dual_tree}, removing the balanced cut edge of a tree $T\in \mathcal{S}_n$ results in a unique loop $\gamma$ in $\G^*_n$ at $v_\infty$.  Our strategy is to first count the possibilities for $\gamma$ and then count the possibilities for $T\in \mathcal{S}_n$ associated with each possibility for $\gamma$.  

We claim first that there are exactly $n$ possibilities for $\gamma$ (loops in $\G_n^*$ that divide the vertices of $\G_n$ equally) which can be easily enumerated.  Figure~\ref{F:channels} illustrates in pink the possibilities for $\gamma$ when $m=5$ (including $n=11$ on the left and $n=10$ on the right).  Each ``$\times 2$'' multiplier symbol in the figure indicates that the reflection of the configuration over the horizontal centerline is an additional option for $\gamma$.  It is straightforward to generalize this scheme to arbitrary $m$.

Enumerate the configurations for $\gamma$ (modulo horizontal symmetry) as $\gamma_0,\dots,\gamma_m$, where the meaning of the index $i$ of $\gamma_i$ is as exemplified in the figure. We claim that the $i^\text{th}$ configuration is associated with exactly $$b_i\cdot \text{length}(\gamma_i)\cdot T_{m-i}^2$$ balanced trees in $\G_n$, where $b_i\in\{1,2\}$ is the multiplier; that is, $b_i=2$ if $\gamma$ is different from its horizontal reflection.  Here the factor $\text{length}(\gamma_i)$ counts the number of placements of the balanced cut edge, while the factor $T^2_{m-1}$ counts the number of ways to draw spanning trees in end blocks (which are highlighted yellow in Figures~\ref{F:split_var} and~\ref{F:channels}).  In summary,
$$S_n = \text{length}(\gamma_m)+ \sum_{i=0}^{m-1} b_i\cdot \text{length}(\gamma_i)\cdot T_{m-i}^2.$$
Note that $\text{length}(\gamma_m)=n$, while for $i\in\{0,\dots m-1\}$ we have 
$$\text{length}(\gamma_i) = \begin{cases}
2i+3 & \text{if } n \text{ is odd} \\
2i+2 & \text{if } n \text{ is even.}
\end{cases}$$
Making these substitutions completes the proof.
\end{proof}

\begin{figure}[bht!]\centering
\includegraphics[width=5in]{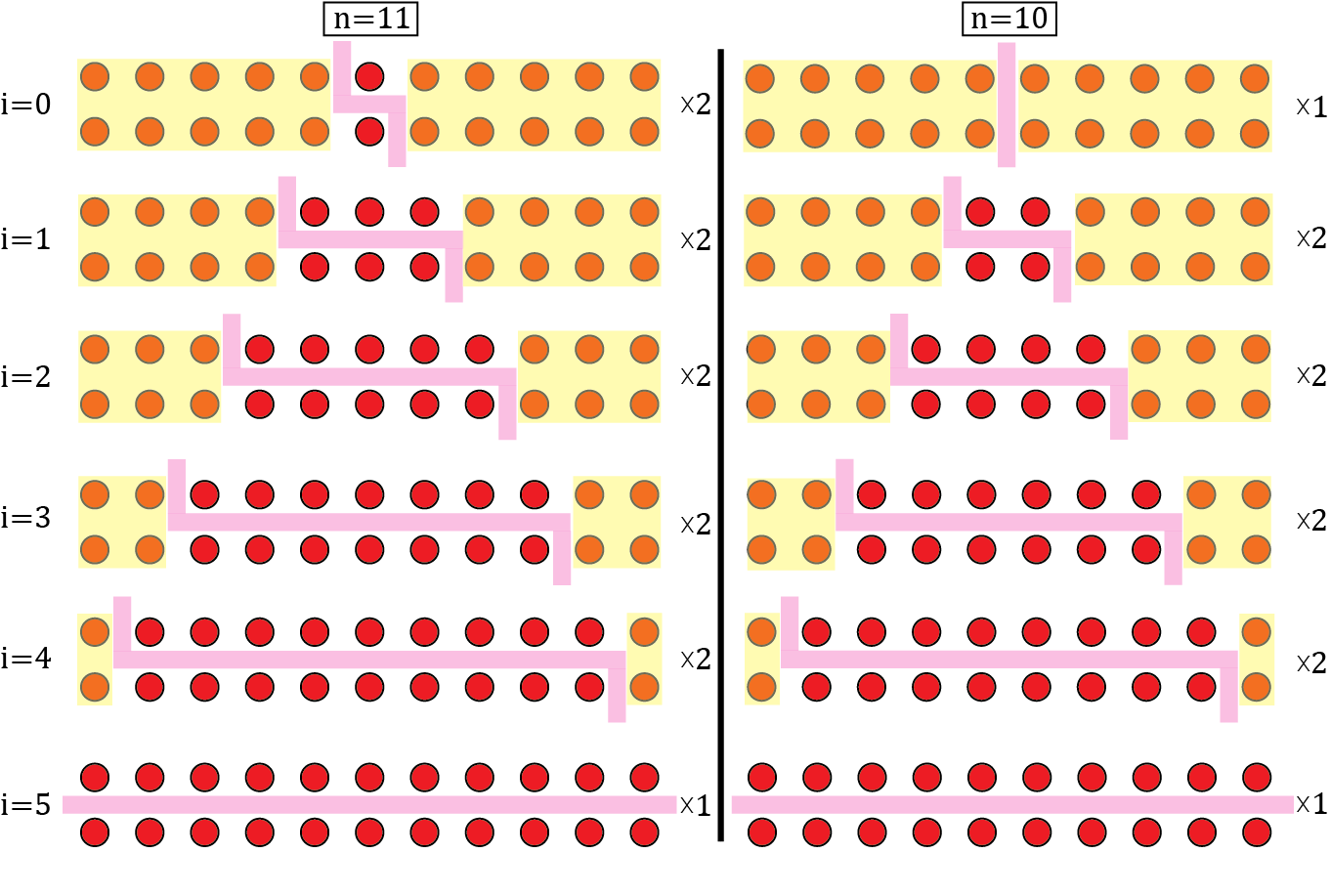}
\caption{The $n$ loops in $\G_{n}^*$ that divide the vertices of $\G_n$ equally are enumerated for $n=11$ (left) and $n=10$ (right).}\label{F:channels}
\end{figure}

\section{Proof of Theorem~\ref{T:main}}
The solution to Equation~\ref{E:Raff} (with initial conditions $T_{1} = 1$, $T_{2} = 4$) is
\begin{equation}\label{E:Raff_sol}
T_k = a\left(r^k - r^{-k} \right),\qquad
\text{ where }a=\frac{1}{2\sqrt{3}}, \text{ and } r = 2+\sqrt{3}.
\end{equation}
In this section, we combine Lemma~\ref{L:even_odd} with Equation~\ref{E:Raff_sol} to prove Theorem~\ref{T:main}.

\begin{proof}[Proof of Theorem~\ref{T:main}]

\textbf{(Odd case):} We first consider the odd case $n=2m+1$.  For each fixed index $i\in\{1,...,m\}$, notice that
\begin{align*}
\frac{T_{m-i}^2}{T_n} 
  & =  a\cdot \frac{\left( r^{m-i}-r^{-(m-i)}\right)^2}{r^{2m+1}-r^{-(2m+1)}} \\
  &  =   a\cdot \frac{r^{2m-2i} + r^{-2m+2i} +2 }{r^{2m+1}-r^{-2m-1}}\\
  & =  a\cdot \frac{1+r^{-4m+4i}+2r^{-2m+2i}}{r^{2i+1}-r^{-4m+2i-1}}\rightarrow \frac{a}{r^{2i+1}}
\end{align*}
where ``$\rightarrow$'' indicates the limit as $m\rightarrow\infty$.  Combining this with Lemma~\ref{L:even_odd} gives:
\begin{align}\label{E:today}
\frac{S_n}{T_n} & = \frac{n}{T_n} + \sum_{i=0}^{m-1} \underbrace{(6 + 4i)\cdot \frac{T_{m-i}^2}{T_{2m+1}}}_{\text{denoted } Y(i,m)}
    \rightarrow \sum_{i=0}^{\infty} (6 + 4i)\cdot \frac{a}{r^{2i+1}}
    =  \frac{a}{r}\sum_{i=0}^{\infty} (6 + 4i)\cdot r^{-2i}.
\end{align}
Setting $x = r^{-2} = 7 - 4\sqrt{3}$, this sum converges to:
\begin{align*}
\sum_{i=0}^\infty (6 + 4i) x^i = \frac{6}{1 - x} + \frac{4x}{(1 - x)^2}
 =  \left(2\sqrt{3} + 3\right) + \left(  \frac{1}{3}\right) =  2\sqrt{3} + \frac{10}{3}.
\end{align*}
Making this substitution and simplifying yields the desired result for the odd case.

It remains to justify the limit claim in Equation~\ref{E:today}, which asserts that 
\begin{equation}\label{E:tomorrow}
\lim_{m\rightarrow\infty}\sum_{i\geq 0} Y(i,m)=\sum_{i \geq 0}\lim_{m\rightarrow\infty}Y(i,m),
\end{equation}
where we can consider all sums to be infinite by setting $Y(i,m)=0$ if $i\geq m$.
To justify Equation~\ref{E:tomorrow} via the Dominated Convergence Theorem, we must construct a sequence $M_i$ such that $\sum_i M_i<\infty$ and $Y(i,m)<M_i$ for all $i\geq 0$ and all $m\geq 1$.  The sequence $M_i=\frac{6+4i}{3^{2i+1}}$ fits the bill because 
\begin{equation}\label{E:nextweek}
\frac{T^2_{m-i}}{T_{2m+1}}\leq \frac{1}{T_{2i+1}} \leq \frac{1}{3^{2i+1}}.
\end{equation}
Equation~\ref{E:nextweek} follows from the more general fact that $T_a\cdot T_b\cdot T_c\leq T_{a+b+c}$, which can be justified algebraically or by observing that there are multiple ways to add edges to a triple in $\mathcal{T}_a\times \mathcal{T}_b\times\mathcal{T}_c$ to generate a tree in $\mathcal{T}_{a+b+c}$.

\textbf{(Even case):} We next consider the even case $n=2m$.  Similar to the previous argument, for each fixed index $i\in\{0,...,m\}$, we have that
\begin{align*}
\frac{T_{m-i}^2}{T_n} 
 \rightarrow \frac{a}{r^{2i}} = a\cdot x^i.
\end{align*}
Combining this with Lemma~\ref{L:even_odd} gives:
\begin{align*}
\frac{S_n}{T_n} & = \frac{n}{T_n} +2\cdot \frac{T_m^2}{T_n} + \sum_{i=1}^{m-1} (4 + 4i)\cdot \frac{T_{m-i}^2}{T^n} \\
    & \rightarrow 2a + a \sum_{i=1}^{\infty} (4 + 4i)\cdot x^i\\
    & = 2a+ a\left( \frac{4x}{1-x} + \frac{4x}{(1-x)^2} \right),
\end{align*}
which simplifies to the desired result.
\end{proof}
\section{Conclusion and future directions}\label{S:conclusion}
We have obtained explicit formulas and limits for the probability that a uniformly random spanning tree of the $2$-by-$n$ grid $\G_n$ is $2$-balanced.  In other words, for this simple class of graphs, we obtained (in the $k=2$ case) exact solutions for the probabilities for which bounds were obtained in~\cite{Cannon_Pegden_Tucker-Foltz_2024}.

Future research could study the probability that a spanning tree of $\G_n$ is $k$-balanced for $k>2$, or could consider the $3$-by-$n$ grid.  It is also important to understand the portion of trees that are \emph{nearly balanced}, which means there is an edge whose removal partitions the vertices into two sets whose sizes differ by at most a fixed tolerance.  

Another future avenue is to study the same question with respect to the MST (minimal spanning tree) distribution on the set of spanning trees of $\G_n$.  In other words, what is the probability that a spanning tree $T$ of $\G_n$ is balanced if, instead of choosing $T$ uniformly (from the ``UST distribution''), it is constructed one edge at a time, with each next edge uniformly randomly selected, along the way rejecting additions that would create cycles?  This is the same as applying Kruskal's algorithm to find a minimum spanning tree with respect to random edge weights.

Table~\ref{T:MST} summarizes for each $n\in\{2,...,19\}$ the portion of trees that are balanced with respect to both the UST and MST distribution.  The UST probabilities in this table are computed as exact fractions via Equation~\ref{E:Raff} and Lemma~\ref{L:even_odd}, and then rounded to $6$ decimals.  The MST probabilities for $n\leq 5$ are computed as exact fractions by enumerating all permutations of the edges of $\G_n$ and tallying the portion of these permutations (considered as edge weights) for which Kruskal's algorithm yields a balanced tree.  The ``$\sim$'' symbol indicates that for $n\geq 6$, the MST probabilities are \emph{approximated} from a sample of one million permutations of the edges of $\G_n$ (because there are too many permutations to loop over all of them).   

It is interesting that, in both the even and odd cases, the UST probabilities converge so quickly to the limit values reported in Theorem~\ref{T:main}.

For odd $n$, the MST values in the table are perhaps unsurprising.  For example, when $n=3$, the exact MST probability of balance ($4/7$) is a bit smaller than the exact UST probability of balance ($6/10$).  The numerical results are consistent with the guess that $MST<UST$ for all odd $n\geq 3$.  It is not clear whether the MST and UST distributions have the same limit in the odd case as $n\rightarrow\infty$.

For even $n$, the MST values in the table are quite surprising.  In the case $n=4$, the exact MST probability of balance ($248/315$) is \emph{larger} than the exact UST probability of balance ($11/14$).  Moreover, given the very large sample size of 1 million permutations, the table provides overwhelming statistical evidence that $MST>UST$ for all $n\in\{6,8,10,12,14,16\}$ as well (with $p$-values smaller than $10^{-100}$).  The table might lead one to guess that the MST limit is strictly larger than the UST limit in the even case as $n\rightarrow\infty$.

This data is surprising because the prevailing expectation was that the MST probability of balance should be $\leq$ the UST probability of balance, at least for natural classes of graphs like grids.

\captionsetup[subtable]{labelformat=empty}
\begin{table}[ht]\label{T:MST}
    \centering
    \begin{subtable}[t]{0.45\textwidth}
        \centering
        \caption{EVEN $n$}
        \begin{tabular}{ccc}
            \toprule
            $n$ & UST & MST \\
            \midrule
            2 & $1$ & $1$ \\
            4 & $\frac{11}{14}\approx 0.785714$ & $\frac{248}{315} \approx 0.787302$ \\
            6 & $0.764103$ & $\sim0.779764$ \\
            8 & $0.762887$ & $\sim0.781753$ \\
            10 & $0.762880$ & $\sim0.783348$ \\
            12 & $0.762890$ & $\sim0.783346$ \\
            14 & $0.762891$ & $\sim0.783693$ \\
            16 & $0.762892$ & $\sim0.783564$ \\
            18 & $0.762892$ & $\sim0.783841$ \\
            \bottomrule
        \end{tabular}
    \end{subtable}
    \hfill
    \begin{subtable}[t]{0.45\textwidth}
        \centering
        \caption{ODD $n$}
        \begin{tabular}{ccc}
            \toprule
            $n$ & UST & MST \\
            \midrule
            3 & $\frac{6}{10}=0.6$ & $\frac{4}{7}\approx 0.571429$ \\
            5 & $\frac{111}{209}\approx 0.531100$ & $\frac{70052}{135135}\approx 0.518385$ \\
            7 & $0.525936$ & $\sim 0.522493$ \\
            9 & $0.525761$ & $\sim0.523989$ \\
            11 & $0.525778$ & $\sim0.524681$ \\
            13 & $0.525783$ & $\sim0.524247$ \\
            15 & $0.525783$ & $\sim0.524980$ \\
            17 & $0.525783$ & $\sim0.524095$ \\
            19 & $0.525783$ & $\sim0.524333$ \\
            \bottomrule
        \end{tabular}
    \end{subtable}
    \caption{The probability that a (UST or MST) random spanning tree of $\G_n$ is balanced.  All values are exact fractions rounded to $6$ digits, except where ``$\sim$'' indicates an approximation from one million random trees.}
\label{T:MST}\end{table}

\bibliographystyle{alpha}  
\bibliography{bibliography} 
\end{document}